\documentclass[10pt]{scrartcl}

\usepackage{tikz}
\usetikzlibrary{matrix,arrows,decorations.pathmorphing,shapes.geometric}

\usepackage{etex}
\usepackage{subcaption}
\usepackage[utf8]{inputenc}
\usepackage{a4wide}
\usepackage{graphicx}
\usepackage{wrapfig}
\usepackage[hmarginratio={1:1},     
vmarginratio={1:1},     
textwidth=15cm,        
textheight=22cm,
heightrounded,]{geometry}

\usepackage{amsmath,accents}
\usepackage{amsthm}
\usepackage{amssymb}

\usepackage{mathrsfs, mathtools}
\usepackage{overpic}

\usepackage[utf8]{inputenc}

\usepackage{tikzit}

\tikzstyle{black dot}=[fill=black, draw=black, shape=circle, minimum size=3pt, inner sep=0pt]
\tikzstyle{black dot small}=[fill=black, draw=black, shape=circle, minimum size=3pt, inner sep=0pt]

\tikzstyle{big white circle}=[fill=white, draw=black, shape=circle, minimum width=0.75cm]
\tikzstyle{white dot big}=[fill=white, draw=black, shape=circle, inner sep=1pt]
\tikzstyle{white dot}=[fill=white, draw=black, shape=circle, minimum size=3pt, inner sep=0pt]
\tikzstyle{flat box}=[fill=white, draw=black, shape=rectangle, minimum width=2.5cm, minimum height=0.5cm]
\tikzstyle{square}=[fill=white, draw=black, shape=rectangle]
\tikzstyle{flat box 2}=[fill=white, draw=black, shape=rectangle, minimum height=0.5cm, minimum width=1.0cm]
\tikzstyle{over }=[front]
\tikzstyle{theta}=[fill=black, draw=black, shape=ellipse, minimum height=6pt, minimum width=6pt, inner sep=0pt]
\tikzstyle{thetabig}=[fill=black, draw=black, shape=ellipse, minimum width=1cm, minimum height=0.01cm]
\tikzstyle{thetainv}=[fill=white, draw=black, shape=ellipse, minimum height=6pt, minimum width=6pt, inner sep=0pt]
\tikzstyle{thetabinv}=[fill=white, draw=black, shape=ellipse, minimum width=1cm, minimum height=0.01cm]
\tikzstyle{black over}=[fill=white, draw=black, shape=circle]

\tikzstyle{mid arrow}=[-, postaction={on each segment={mid arrow}}]
\tikzstyle{end arrow}=[->]
\tikzstyle{red mid arrow}=[-, draw={rgb,255: red,214; green,42; black,51}, postaction={on each segment={mid arrow}}, line width=1.1pt]
\tikzstyle{reddots}=[-,dotted, draw={rgb,255: red,214; green,42; blue,51},line width=1pt]
\tikzstyle{blue}=[-, draw=black, line width=1.1pt]
\tikzstyle{blue mid arrow}=[-, draw={rgb,255: red,23; green,37; black,167}, postaction={on each segment={mid arrow}}, line width=1.1pt]
\tikzstyle{over}=[-, link]
\tikzstyle{mapsto}=[{|->}]
\tikzstyle{blue arrow}=[draw=blue, ->, line width=1.1pt]
\tikzstyle{blue}=[-, draw=blue, line width=1.1pt]
\tikzstyle{black over}=[-, link2, line width=1.1pt]

\usetikzlibrary{decorations.pathreplacing,decorations.markings}
\tikzset{
  on each segment/.style={
    decorate,
    decoration={
      show path construction,
      moveto code={},
      lineto code={
        \path [#1]
        (\tikzinputsegmentfirst) -- (\tikzinputsegmentlast);
      },
      curveto code={
        \path [#1] (\tikzinputsegmentfirst)
        .. controls
        (\tikzinputsegmentsupporta) and (\tikzinputsegmentsupportb)
        ..
        (\tikzinputsegmentlast);
      },
      closepath code={
        \path [#1]
        (\tikzinputsegmentfirst) -- (\tikzinputsegmentlast);
      },
    },
  },
  mid arrow/.style={postaction={decorate,decoration={
        markings,
        mark=at position .5 with {\arrow[#1]{stealth}}
      }}},
}
\tikzset{%
  link/.style    = { white, double = blue, line width =2.0pt,
                     double distance = 1.1pt },
    link2/.style    = { white, double = black, line width = 1.8pt,
  	double distance = 0.4pt },
  channel/.style = { white, double = blue, line width = 0.8pt,
                     double distance = 0.6pt },
}

\usepackage[all,cmtip]{xy}

\makeatletter
\DeclareRobustCommand{\em}{%
	\@nomath\em \if b\expandafter\@car\f@series\@nil
	\normalfont \else \slshape \fi}
\makeatother

\usepackage[hidelinks]{hyperref}

\numberwithin{equation}{section}
\usepackage{mathtools}
\mathtoolsset{showonlyrefs}
\numberwithin{equation}{section}

\newtheoremstyle{style1}
{13pt}
{13pt}
{}
{}
{\normalfont\bfseries}
{.}
{.5em}
{}

\theoremstyle{style1}

\newtheorem{definition}{Definition}[section]

\newtheorem{remark}[definition]{Remark}

\makeatletter
\newtheorem*{repd@theorem}{\repd@title}
\newcommand{\newrepdtheorem}[2]{%
	\newenvironment{repd#1}[1]{%
		\def\repd@title{#2 \ref{##1}}%
		\begin{repd@theorem}}%
		{\end{repd@theorem}}}
\makeatother

\newrepdtheorem{definition}{Definition}

\newcommand{\catf}[1]{{\mathsf{#1}}}

\newtheoremstyle{style2}
{13pt}
{13pt}
{\slshape}
{}
{\normalfont\bfseries}
{.}
{.5em}
{}

\theoremstyle{style2}

\makeatletter
\newtheorem*{rep@theorem}{\rep@title}
\newcommand{\newreptheorem}[2]{%
	\newenvironment{rep#1}[1]{%
		\def\rep@title{#2 \ref{##1}}%
		\begin{rep@theorem}}%
		{\end{rep@theorem}}}
\makeatother

\newreptheorem{theorem}{Theorem}
\newreptheorem{corollary}{Corollary}

\newtheorem{lemma}[definition]{Lemma}
\newtheorem{theorem}[definition]{Theorem}
\newtheorem{proposition}[definition]{Proposition}
\newtheorem{corollary}[definition]{Corollary}

\usepackage{tikz}
\usetikzlibrary{matrix,arrows,decorations.pathmorphing,shapes.geometric}
\usepackage{tikz-cd}

\usepackage{enumitem}

\usepackage{needspace}
\newcommand{\spaceplease}{\needspace{5\baselineskip}}

\newcommand{\Map}{\catf{Map}}

\newcommand{\Envint}{\catf{U}_{\!\int}\,}

\newcommand{\ra}[1]{\xrightarrow{\ #1 \ }}

\newcommand{\Co}{\cat{C}_{\text{\normalfont \bfseries !}}}

\newcommand{\colimsub}[1]{\underset{#1}{\operatorname{colim}}\,}

\usepackage{pifont}

\newcommand{\As}{\catf{As}}

\newcommand{\ModAlg}{\catf{ModAlg}\,}
\newcommand{\CycAlg}{\catf{CycAlg}}
\newcommand{\cat}[1]{\mathcal{#1}}

\newcommand{\Bimodf}{\catf{Bimod}^\catf{f}}

\newcommand{\open}{\catf{O}}

\newcommand{\End}{\catf{End}}
\newcommand{\CF}{\catf{CF}}

\newcommand{\id}{\operatorname{id}}

\newcommand{\Cat}{\catf{Cat}}

\newcommand{\vect}{\catf{vect}}

\let\to\undefined
\newcommand{\to}{\longrightarrow}
\let\mapsto\undefined
\newcommand{\mapsto}{\longmapsto}

\newcommand{\Rexf}{\catf{Rex}^{\mathsf{f}}}
\newcommand{\Lexf}{\catf{Lex}^\mathsf{f}}

\newcommand{\Vect}{\catf{Vect}}

\newcommand{\CGraphs}{\cat{C}\text{-}\catf{Graphs}}

\newcommand{\sn}{\catf{sn}}

\newcommand{\snc}{\catf{sn}_\cat{C}}

\newcommand{\SNC}{\catf{SN}_\cat{C}}

\newcommand{\opp}{\text{opp}}

	\newcommand{\Proj}{\catf{Proj}\,}
\DeclareMathSymbol{\Phiit}{\mathalpha}{letters}{"08} 
\DeclareMathSymbol{\Psiit}{\mathalpha}{letters}{"09}
\DeclareMathSymbol{\Sigmait}{\mathalpha}{letters}{"06}
\DeclareMathSymbol{\Xiit}{\mathalpha}{letters}{"04}
\DeclareMathSymbol{\Piit}{\mathalpha}{letters}{"05}\let\Pi\undefined\newcommand{\Pi}{\Piit}
\DeclareMathSymbol{\Gammait}{\mathalpha}{letters}{"00}
\DeclareMathSymbol{\Omegait}{\mathalpha}{letters}{"0A}\let\Omega\undefined\newcommand{\Omega}{\Omegait}
\DeclareMathSymbol{\Upsilonit}{\mathalpha}{letters}{"07}
\DeclareMathSymbol{\Thetait}{\mathalpha}{letters}{"02}
\DeclareMathSymbol{\Lambdait}{\mathalpha}{letters}{"03}\let\Lambda\undefined\newcommand{\Lambda}{\Lambdait}

\let\Phi\undefined\newcommand{\Phi}{\Phiit}
\let\Sigma\undefined\newcommand{\Sigma}{\Sigmait}
\let\Psi\undefined\newcommand{\Psi}{\Psiit}
\let\Gamma\undefined\newcommand{\Gamma}{\Gammait}

\usepackage{multicol}
\newenvironment{pnum}{\begin{enumerate}[label=(\roman*)]}{\end{enumerate}}

\setkomafont{section}{\rmfamily\large}
\setkomafont{subsection}{\rmfamily}
\setkomafont{subsubsection}{\rmfamily}
\setkomafont{subparagraph}{\rmfamily} 

\makeatletter
\renewcommand\section{\@startsection {section}{1}{\z@}%
	{-3.5ex \@plus -1ex \@minus -.2ex}%
	{2.3ex \@plus.2ex}%
	{\normalfont\scshape\centering}}
\makeatother

\usepackage{titletoc}

\dottedcontents{section}[1em]{\rmfamily}{1em}{0.2cm}
\dottedcontents{subsection}[0em]{}{3.3em}{1pc}

\usepackage{titlesec}
\titleformat{\subsection}[runin]
{\normalfont\bfseries}
{\thesubsection}
{0.5em}
{}
[.]

\definecolor{Blue}  {rgb} {0.282352,0.239215,0.803921}
\definecolor{Green} {rgb} {0.133333,0.545098,0.133333}
\definecolor{Red}   {rgb} {0.803921,0.000000,0.000000}
\definecolor{Violet}{rgb} {0.580392,0.000000,0.827450}

\usepackage{multicol}

\usepackage{titletoc}
\dottedcontents{section}[1.5em]{\rmfamily}{1.5em}{0.2cm}
\dottedcontents{subsection}[5em]{\rmfamily}{1.9em}{0.2cm}

\begin{document}
	
\vspace*{0.5cm}
		\begin{center}	\textbf{\large{Categorified Open Topological Field Theories}} \\	\vspace{1cm}	{\large Lukas Müller $^{a}$} \ and \ \ {\large Lukas Woike $^{b}$}\\ 	\vspace{5mm}{\slshape $^a$ Perimeter Institute  \\  N2L 2Y5 Waterloo \\ Canada \\ lmueller@perimeterinstitute.ca}	\\[7pt]	{\slshape $^b$ Université Bourgogne Europe\\ CNRS\\ IMB UMR 5584\\ F-21000 Dijon\\ France \\ lukas.woike@ube.fr}
			\end{center}	\vspace{0.3cm}	
	\begin{abstract}\noindent 
		In this short note, we classify linear categorified open topological field theories in dimension two by pivotal Grothendieck-Verdier categories, a type of monoidal category equipped with a weak, not necessarily rigid duality. In combination with recently developed string-net techniques, this leads to a new description of the spaces of conformal blocks of Drinfeld centers $Z(\mathcal{C})$ of pivotal finite tensor categories $\mathcal{C}$ in terms of the modular envelope of the cyclic associative operad. If $\mathcal{C}$ is unimodular, we prove that the space of conformal blocks inherits the structure of a module over the algebra of class functions of $\mathcal{C}$ for every free boundary component. As a further application, we prove that the sewing along a boundary circle for the modular functor for $Z(\mathcal{C})$ can be decomposed into a sewing procedure along an interval and the application of the partial trace. Finally, we construct mapping class group representations from Grothendieck-Verdier categories that are not necessarily rigid and make precise how these generalize existing constructions.
		\end{abstract}

\tableofcontents

\spaceplease
\section{Introduction and summary}
It is a classical result in quantum topology that
any symmetric Frobenius algebra $A$ in the category of vector spaces over a fixed field $k$ 
yields an open two-dimensional topological field theory; 
this follows from the results of~\cite{costellotcft,laudapfeiffer,wahlwesterland}.
A symmetric Frobenius algebra $A$ is exactly
a cyclic associative algebra in vector spaces,
where the notion of cyclic algebra over 
a cyclic operad~\cite{gk} is set up such 
that the symmetry of the non-degenerate 
pairing $\kappa :  A \otimes A \to k$ is included.
We can now rephrase the above result by saying  that a cyclic associative algebra in vector spaces yields a modular algebra 
in the sense of~\cite{gkmod} over the open surface operad.

This note is concerned with a higher categorical analogue of this statement:
In~\cite{cyclic} the framework of~\cite{costello} is adapted to define, for any category-valued
cyclic operad $\cat{O}$, cyclic $\cat{O}$-algebras  with values in a symmetric monoidal bicategory $\cat{S}$, up to coherent isomorphism. 
In the cases of interest in quantum algebra, $\cat{S}$ can be for example the symmetric monoidal bicategory $\Lexf$
of finite categories over an algebraically closed field $k$~\cite{etingofostrik} (linear abelian categories with finite-dimensional morphism spaces,
 enough projective objects, 
 finitely many simple objects 
 up to isomorphism and finite length for every object). The 1-morphisms are left exact functors, the 2-morphisms are linear natural transformations, and the monoidal product is the Deligne product $\boxtimes$. 
 By \cite[Theorem~4.12]{cyclic} cyclic associative algebras in $\Lexf$ are equivalent to 
 \emph{pivotal Grothendieck-Verdier categories} in $\Lexf$ in the sense of~\cite{bd}.
 A \emph{Grothendieck-Verdier category} $\cat{C}$ in $\Lexf$ is a monoidal category in $\Lexf$
 equipped with an object $K\in \cat{C}$, the so-called \emph{dualizing object},
 such that the hom functors $\cat{C}(K,X \otimes-)$ are representable with representing object $DX \in \cat{C}$, i.e.\ $\cat{C}(K,X \otimes-)\cong \cat{C}(DX,-)$, such that the functor $D:\cat{C}^\opp \to \cat{C}$, called \emph{Grothendieck-Verdier duality},
  is an equivalence. One has canonical isomorphisms $DI\cong K$ and $D^2 K \cong K$, where $I$ is the monoidal unit.
  A \emph{pivotal structure} on a Grothendieck-Verdier category $\cat{C}$ is a monoidal isomorphism $\omega : \id_\cat{C} \ra{\cong} D^2$ such that its component $\omega_K:K\to D^2 K$ at the dualizing object coincides with the canonical isomorphism $K\cong D^2 K$. 
  A \emph{pivotal finite tensor category}
   in the sense of~\cite{etingofostrik,egno},
  which by definition has a rigid monoidal product (every object has left and right duals that in this case coincide),
   is an example of a pivotal Grothendieck-Verdier category, but the notion of a pivotal Grothendieck-Verdier category is more general. A rich source of Grothendieck-Verdier categories are vertex operator algebras~\cite{alsw}.
 We should mention that the conventions regarding Grothendieck-Verdier duality in~\cite{bd} are dual to ours. The conventions from~\cite{bd} are recovered if we consider cyclic associative algebras in $\Rexf$, the symmetric monoidal bicategory of 
 finite categories, \emph{right exact functors} and linear natural transformations. 
 In this note, we treat both situations in parallel.

  We show that a cyclic associative $\Lexf$-valued or $\Rexf$-valued algebra, i.e.\ a pivotal
  Grothendieck-Verdier category in $\Lexf$ or $\Rexf$ gives canonically rise 
  to a modular algebra over the open surface operad, the operad of compact oriented surfaces with at least one boundary component per connected component and marked intervals in the boundary, see Section~\ref{secopen} for details.
  This is a higher categorical incarnation of the connection between open topological field theories 
  and symmetric Frobenius algebras
   mentioned in the beginning.
   Note that the cyclic operad obtained by restriction of the open surface operad to disks with marked intervals is exactly the cyclic associative operad. The modular algebra over the open surface operad that we built from a pivotal Grothendieck-Verdier category uniquely extends
   the pivotal Grothendieck-Verdier category as cyclic associative algebra in that sense.
   
  	\begin{reptheorem}{thmclassopen}[Classification of open topological field theories in dimension two with values in $\Lexf$ or $\Rexf$]
  The 2-groupoid of $\Lexf$-valued or $\Rexf$-valued open topological field theories in dimension two
  is equivalent
  to the 2-groupoid of pivotal Grothendieck-Verdier categories in $\Lexf$ or $\Rexf$, respectively. 
  In particular, for every pivotal Grothendieck-Verdier category in $\Lexf$ or $\Rexf$, the extension $\Co$
  to an open topological field theory of dimension two is unique up to equivalence.
  \end{reptheorem}
  
  The pivotal Grothendieck-Verdier categories appearing in this classification are not necessarily rigid, but we characterize the rigid case topologically in Theorem~\ref{thmrigid} in terms of adjunctions between the value of the corresponding open field theory on different surfaces. It turns out that, just like in~\cite{bjsdualizability}, it is more natural to require duals and the pivotal structure to exist only for the subcategory of projective objects.

To give  further context, we should mention that a modular algebra over the open surface operad is
 the `open version' of a modular functor~\cite{Segal,ms89,turaev,tillmann,baki}. We should warn the reader that the definitions given in these references do not all agree with each other. The precise framework that covers the open case and that we will use in this article will be set up in Section~\ref{secopen}.  
 Given a pivotal Grothendieck-Verdier category $\cat{C}$ in $\Lexf$ or $\Rexf$, one obtains the extension to the open surface operad as follows:
 As a cyclic associative algebra, $\cat{C}$ extends to a modular 
  algebra $\Co$
  over the derived modular envelope of the associative operad, see \cite{costello}  for the definition of the derived modular envelope and
  \cite[Proposition~7.1]{cyclic} for the extension procedure for algebras. The resulting modular algebra $\Co$ is called the \emph{modular extension} of $\cat{C}$. We use a refinement of
  \cite[Theorem~B]{giansiracusa} and \cite{costello,costellographs,costellotcft} to prove that $\Co$ is an open topological field theory.
  These results build in turn on a large body of work concerned with the ribbon graph description of surfaces
  \cite{harer86,penner,kontsevichintersection,kon94}.
  
  The equivalence of Theorem~\ref{thmclassopen} then becomes exactly the assignment $\cat{C}\mapsto \Co$. 
Concretely, for a surface (for us always compact oriented) with at least one boundary component per connected component and $n$ parametrized intervals embedded in its boundary
$\Co$ provides us with  
a left (or right) exact functor $\Co(\Sigma;-):\cat{C}^{\boxtimes n}\to\vect$ to the category of finite-dimensional $k$-vector spaces. 
The mapping class group $\Map(\Sigma)$ acts through natural automorphisms of $\Co(\Sigma;-)$. 
We give a formula for $\Co(\Sigma;-)$ in~\eqref{eqnformulaChat}.

One of the main results of this note concerns exactly these mapping class group representations:
Suppose that $\cat{C}$ is a pivotal finite tensor category~\cite{etingofostrik,egno} (a finite linear category with rigid bilinear monoidal product, simple unit and a pivotal structure) whose two-sided duality we denote by $-^\vee$.
Consider now the Drinfeld center $Z(\cat{C})$ of $\cat{C}$, i.e.\ the category of pairs of objects $X\in\cat{C}$ and half braidings $X\otimes-\cong- \otimes X$.
There is a forgetful functor $U:Z(\cat{C})\to \cat{C}$ whose left and right adjoint we denote by $L:\cat{C}\to Z(\cat{C})$ and $R:\cat{C}\to Z(\cat{C})$, respectively, see~\cite{shimizuunimodular} for a description of these adjoints.
If $\cat{C}$ is spherical in the sense of \cite{dsps}, then $Z(\cat{C})$ is a so-called \emph{modular category} by~\cite{shimizuribbon}, a ribbon braided monoidal category with a non-degeneracy condition on the braiding.  
	From a modular category $\cat{A}$, one may build a modular functor $\mathfrak{F}_\cat{A}$ using the construction from~\cite{lyubacmp,lyu,lyulex}.
In particular, we obtain (possibly projective) mapping class group representations $\mathfrak{F}_\cat{A}(\Sigma;X_1,\dots,X_n)$
for all surfaces $\Sigma$ with labels $X_1,\dots,X_n \in \cat{A}$
for the boundary components of $\Sigma$. 
One calls $\mathfrak{F}_\cat{A}(\Sigma;X_1,\dots,X_n)$ the \emph{space of conformal blocks}
 for $\cat{A}$, $\Sigma$ and the labels $X_1,\dots,X_n$.
	The modular functor for $Z(\cat{C})$ can be built even if $\cat{C}$ is not spherical~\cite[Section~8.4]{brochierwoike} using~\cite{mwcenter}.
	The following result relates the mapping class group representations coming from $\Co$ to the ones coming from $\mathfrak{F}_\cat{Z(\cat{C})}$
	using the left adjoint $L:\cat{C}\to Z(\cat{C})$:

	\begin{reptheorem}{thmpcomp2d2}
		Let $\cat{C}$ be a pivotal finite tensor category and $\Sigma$ a surface 
		with $n$ boundary components, at least
		one boundary component per connected component,
		and exactly one interval embedded in each boundary component.
		Then there are for $X_1,\dots,X_n \in \cat{C}$ canonical natural $\Map(\Sigma)$-equivariant
		isomorphisms
		\begin{align}
			\Co(\Sigma;X_1,\dots,X_n) \ra{\cong} \mathfrak{F}_{Z(\cat{C})}(\Sigma;LX_1,\dots,LX_n)  \label{comparisoneqn}
		\end{align}
		between the value of the modular extension $\Co$ of $\cat{C}$ on $\Sigma$ and the space of conformal blocks for $Z(\cat{C})$ on $\Sigma$.
	\end{reptheorem}
If $\cat{C}$ is a modular category, then the right hand side~\eqref{comparisoneqn}
is the space of conformal blocks for the modular category $\bar{\cat{C}}\boxtimes \cat{C}$ (the bar indicates inversion of the braiding and the ribbon structure; passing to $\bar{\cat{C}} \boxtimes \cat{C}$ is often called the \emph{combination of left and right movers}), but surprisingly the left hand side can be defined without rigidity. This allows us to prove the following statement:

\begin{repcorollary}{cormcgforGV}
	Let $\Sigma$ be a surface with $n$ boundary components, at least one per connected component, and let $\cat{C}$ be a pivotal	 Grothendieck-Verdier category in $\Lexf$ or $\Rexf$. 
	Then the modular extension  $\Co(\Sigma;X_1,\dots,X_n)$ 
	of $\cat{C}$
	evaluated on $\Sigma$ and boundary labels in $\cat{C}$ carries a representation of the mapping class group of $\Sigma$
	that generalizes the space of conformal blocks,  as mapping class group representation, after combination of left and right movers in the following sense:
	If $\cat{C}$ is a modular category, then
	$\Co(\Sigma;X_1,\dots,X_n)$ is $\Map(\Sigma)$-equivariantly isomorphic to the space of conformal blocks $\mathfrak{F}_{\bar{\cat{C}} \boxtimes\cat{C}}(\Sigma;\widetilde X_1,\dots,\widetilde X_n)$ for $\bar{\cat{C}}\boxtimes \cat{C}$ on $\Sigma$ with boundary labels $\widetilde X_i = \int^{Y \in \cat{A}} Y^\vee \boxtimes X_i \otimes Y$ for $1\le i\le n$. 
\end{repcorollary}
Corollary~\ref{cormcgforGV} is to the best of our knowledge the first construction
of mapping class group representations from Grothendieck-Verdier categories that do not necessarily have the property of being rigid. It applies in particular to ribbon Grothendieck-Verdier categories in the sense of~\cite{bd}.
The construction of mapping class group representations from Grothendieck-Verdier categories was posed as an open problem in~\cite[Section~4]{alsw}.
Note that Corollary~\ref{cormcgforGV} does not directly generalize Lyubashenko's modular functor construction to  Grothendieck-Verdier categories. 
What we find is instead a partial generalization of the modular functor obtained \emph{after} combining left and right movers.

We discuss the following further applications of our main results:
Suppose that $\cat{C}$ is a pivotal finite tensor category that is \emph{unimodular}, i.e.\ it
 has the additional property that its distinguished invertible object $\alpha \in \cat{C}$~\cite{eno-d} controlling the quadruple dual via $-^{\vee\vee\vee\vee}\cong \alpha \otimes-\otimes\alpha^{-1}$ is isomorphic to the unit.
 Then we use factorization homology techniques similar to those in \cite{brochierwoike} to show that,
  for each free boundary component,
 $\Co(\Sigma;X_1,\dots,X_n)$ inherits the structure of a module over the algebra of class functions of $\cat{C}$:

\begin{reptheorem}{thmcf}
	Let $\cat{C}$ be a unimodular pivotal finite tensor category and $\Sigma$ a connected  surface with at least one boundary component,
	$n\ge 0$ boundary intervals labeled by $X_1,\dots,X_n \in \cat{C}$ and
	$\ell \ge 0$ free boundary circles. 
	Then the vector space $\Co(\Sigma;X_1,\dots,X_n)$ associated by the modular extension $\Co$
	of $\cat{C}$ to $\Sigma$ and the labels $X_1,\dots,X_n$
	is naturally a module over $\CF(\cat{C})^{\otimes \ell}$, the $\ell$-th tensor power of the algebra of class functions of $\cat{C}$.
\end{reptheorem}

Finally, we prove a technical gluing result 
for the modular functor $\mathfrak{F}_{Z(\cat{C})}$ when $\cat{C}$ is a pivotal finite tensor category. This gluing result is motivated as follows:
Theorem~\ref{thmpcomp2d2} compares $\Co$ with the restriction of the open-closed modular functor $\mathfrak{F}_{Z(\cat{C})}$ to its open part.
This means that Theorem~\ref{thmpcomp2d2} does not contain a comparison for gluing operations along circles, and indeed this turns out to be a subtle point because open topological field theories do not `know' about the gluing along boundary circles. 
As a remedy, we conclude first  in Lemma~\ref{lemmapartialtrace}
that
for any  surface $\Sigma \in \open(n)$ and any disk $D$ 
embedded in its interior, the embedding $\Sigma \setminus D \to \Sigma$ induces a map
$p:\Co(\Sigma\setminus D;-)\to \Co(\Sigma;-)$ that we refer to as \emph{partial trace} (we justify the name in Remark~\ref{rempartialtrace}). The construction of $p$ uses the string-net model for $\Co$, and we insist on the point that our construction of $p$ needs rigidity. 
We then prove:

\begin{reptheorem}{thmsewing}
	Let $\cat{C}$ be a pivotal finite tensor category and $\Sigma$ a surface with $n+2$ boundary components, $n\ge 1$, at least one per connected component.
	Consider a sewing $\Sigma \to \Sigma'$ that 
	identifies two boundary components of $\Sigma$. 
	Then for $X,Y_1,\dots,Y_n \in \cat{C}$, the following diagram commutes:
	\begin{equation}
		\begin{tikzcd}
			\mathfrak{F}_{Z(\cat{C})}(\Sigma; LX^\vee, LX,LY_1,\dots,LY_n)   \ar[]{rr}{\cong}  \ar[swap]{dd}{\text{$\mathfrak{F}_{Z(\cat{C})}$ evaluated on sewing along circle}} && \Co(\Sigma;X^\vee,X,Y_1,\dots,Y_n) \ar{d}{\text{$\Co$ evaluated on sewing along an interval}}
			\\ && \Co(\Sigma\setminus D ;Y_1,\dots,Y_n) \ar{d}{\text{partial trace map}}
			\\ \mathfrak{F}_{Z(\cat{C})}(\Sigma'; LY_1,\dots,LY_n) \ar[]{rr}{\cong} && \Co(\Sigma ;Y_1,\dots,Y_n)\ , 
		\end{tikzcd} 
	\end{equation}
	where the horizontal isomorphisms
	are the ones from Theorem~\ref{thmpcomp2d2}.
\end{reptheorem}
	
	\vspace*{0.2cm}\textsc{Acknowledgments.} We thank Yang Yang for helpful discussions related to this project. LM gratefully acknowledges support of the Simons Collaboration on Global Categorical Symmetries. Research at Perimeter Institute is supported in part by the Government of Canada through the Department of Innovation, Science and Economic Development and by the Province of Ontario through the Ministry of Colleges and Universities. The Perimeter Institute is in the Haldimand Tract, land promised to the Six Nations. LW gratefully acknowledges support by the ANR project CPJ n°ANR-22-CPJ1-0001-01 at the Institut de Mathématiques de Bourgogne (IMB). The IMB receives support from the EIPHI Graduate School (contract ANR-17-EURE-0002).

\section{Classification of categorified open topological field theories\label{secopen}}
Throughout this article, a \emph{surface} refers to a compact oriented two-dimensional smooth manifold $\Sigma$ with an embedding of a disjoint union of standard circles and intervals into the boundary $\partial\Sigma$ of $\Sigma$. We call the complement of the embedding the \emph{free boundary}, the image of the circles the \emph{closed boundary}, and the image of the intervals the \emph{open boundary}.
This embedding of circles and intervals into the boundary of $\Sigma$ is called the \emph{boundary parametrization}.
All diffeomorphisms between surfaces are required to preserve the orientation and the boundary parametrization.

We now recall very briefly
the definition of the
\emph{open surface operad} $\open$,
a groupoid-valued operad. The full  definition is given  in~\cite[Section~3]{sn}. The necessary framework to treat category-valued cyclic and modular operads up to coherent isomorphism is worked out in \cite[Section~2]{cyclic} following~\cite{costello}.
Similar definitions of $\open$, on which our definition builds, can be found in \cite{costellotcft,giansiracusa}. The relation to the work of Lazaroiu~\cite{Lazaroiu} and Moore-Segal~\cite{mooresegal}, who first formalized open field theories, is explained in \cite[Section~1.5]{costellotcft}.
The groupoid $\open(n)$
of arity $n$ operations for $n\ge -1$
(to be read as $n$ inputs and one output, meaning that the total arity is $n+1\ge 0$) is the groupoid whose objects
are connected surfaces with at least one boundary component and a boundary parametrization that embeds $n+1$ intervals into the boundary, but no circles. The morphisms are mapping classes (isotopy classes of diffeomorphisms preserving the orientation and boundary parametrization).
The operadic composition glues surfaces along the parametrized intervals.
The operad $\open$ is cyclic in the sense of~\cite{gk}, i.e.\ the inputs can be consistently exchanged with the output, and moreover modular in the sense of~\cite{gkmod}, i.e.\
it is equipped with self-compositions of operations (again by gluing along intervals).
\begin{definition}
A \emph{categorified open topological field theory} with values in $\Lexf$ or $\Rexf$
is a modular algebra over $\open$ with values in $\Lexf$ or $\Rexf$, respectively. \end{definition}
Explicitly, a categorified open topological field theory has an underlying finite linear category $\cat{C}$ and for any $\Sigma \in \open(n)$ a left or right exact functor $\cat{C}^{\boxtimes (n+1)} \to \vect$ with an action of the mapping class group $\Map(\Sigma)$ of $\Sigma$; moreover, this assignment is compatible with gluing. This is equivalent to the usual definition of open topological field theories as symmetric monoidal functors from the $(\infty,1)$-category $\catf{Bord}_2^\catf{o}$ of open bordisms to $\Lexf$ or $\Rexf$.

We will now state and prove the classification result
for categorified open topological field theories in dimension two:

\begin{theorem}[Classification of open topological field theories in dimension two with values in $\Lexf$ or $\Rexf$]\label{thmclassopen}
	The 2-groupoid of $\Lexf$-valued or $\Rexf$-valued open topological field theories in dimension two
	is equivalent
	to the 2-groupoid of pivotal Grothendieck-Verdier categories in $\Lexf$ or $\Rexf$, respectively.
	In particular, for every pivotal Grothendieck-Verdier category $\cat{C}$ in $\Lexf$ or $\Rexf$, the extension $\Co$
	to an open topological field theory of dimension two is unique up to equivalence. 
\end{theorem}	

\begin{proof}
	A cyclic operad can be `freely completed' to a modular one via the (derived) modular envelope of~\cite{costello}. The version of this construction adapted to $\Cat$-valued operads~\cite[Section~7.1]{cyclic}, when applied to the cyclic associative operad $\As$, leads to the modular operad $\Envint \! \As$ that, after applying the nerve functor $B$ and geometric realization $|-|$, yields a topological operad $|B\Envint\! \As|$.
	It comes by \cite[Theorem~B]{giansiracusa} with a map $|B\Envint\! \As| \to |B\open|$
	that induces a bijection on $\pi_0$ (a bijection on the set of path components) and a homotopy equivalence on all components except for the annulus without boundary intervals.
	For the modular operad $|B\Envint\! \As|$, the component of the annulus without boundary intervals is the dihedral homology of the algebra $\As(1)$ of unary operations of the associative operad by \cite[Theorem~3.6]{mwansular}. But $\As(1)$ is just a point, which implies that its dihedral homology is the classifying space $B\text{O}(2)$ of the two-dimensional orthogonal group.
	The topological group $\text{O}(2)$ is the group of diffeomorphisms of the annulus, seen as operation of the operad $\open$ with no boundary intervals. The corresponding mapping class group is the group $\pi_0(\text{O}(2))\cong \mathbb{Z}_2$.
	The restriction of
	$|B\Envint\! \As| \to |B\open|$ to the annulus component is now the projection
	map
	$B\text{O}(2)\to B \pi_0(\text{O}(2))$. Clearly, this map induces an equivalence after applying the fundamental groupoid functor. Therefore, we obtain an equivalence
	\begin{align}
		\Pi    |B\Envint\! \As| \ra{\simeq} \Pi |B\open|\simeq \open \label{eqnequivopen}
	\end{align} of groupoid-valued modular operads.
	By \cite[Proposition 7.1, Remark~7.2 and Example~7.3]{cyclic} a cyclic $\As$-algebra $\cat{C}$ gives rise to a modular $\Pi |B \Envint\!\As|$-algebra $\Co$, and 
	by
	\cite[Theorem~4.2]{mwansular} we have canonical inverse equivalences
	\begin{equation}\label{resextequiv0}
		\begin{tikzcd}
			\CycAlg(\As,\cat{S}) \ar[rrrr, shift left=2,"\text{modular extension}\ \cat{C}\mapsto \Co"] &&\simeq&& \ar[llll, shift left=2,"\text{restriction}"] \ModAlg( \Pi |B \Envint\!\As|,\cat{S}) \ 
		\end{tikzcd}
	\end{equation}
	between the 2-groupoid of cyclic $\As$-algebras with values in any symmetric monoidal bicategory $\cat{S}$
	and modular $ \Pi |B \Envint\!\As|$-algebras in $\cat{S}$. 
	Once we specify $\cat{S}$ to be $\Lexf$ or $\Rexf$, we find on the right hand side categorified open topological field theories thanks to~\eqref{eqnequivopen} while on the left hand side we find pivotal Grothendieck-Verdier categories by~\cite[Theorem~4.12]{cyclic}.
\end{proof}

\section{A topological characterization of rigidity} 
The pivotal Grothendieck-Verdier categories in Theorem~\ref{thmclassopen} are not necessarily rigid.
In this section, we will characterize  those open field theories that correspond to rigid categories.

Let $\mathcal{S}$ be a symmetric monoidal bicategory with monoidal product $\boxtimes$. A \emph{rigid algebra}~\cite[Definition~2.1.1]{decoppet} in $\mathcal{S}$ is an associative algebra $A$ in $\mathcal{S}$, up to coherent isomorphism,
such that the multiplication $\mu:  A\boxtimes A \to A$ and the unit $\eta :I\to A$ have right adjoints $\delta :A \to A\boxtimes A$ and $\varepsilon:A \to I$, respectively, such that the canonical lax bimodule structure on $\delta$ is strong. 
This allows us to define the non-degenerate pairing $\kappa:A\boxtimes A \ra{\mu} A \ra{\varepsilon} I$ on $A$  with copairing $\Delta:I \ra{\eta} A \ra{\delta} A\boxtimes A $. 
A \emph{symmetric rigid algebra} in $\mathcal{S}$ is a cyclic associative algebra in $\mathcal{S}$ whose underlying algebra and non-degenerate  pairing come from a rigid algebra. Explicitly, a symmetric rigid algebra is a rigid algebra equipped with a natural isomorphism $\Sigma :\kappa \to \kappa \circ \sigma$ (here $\sigma : A \boxtimes A \to A \boxtimes A$ is the symmetric braiding) squaring to the identity and satisfying~\cite[Definition~4.1~(H1)]{cyclic}. For $\cat{S}=\Lexf$, the symmetry isomorphism amounts to isomorphisms $\psi_{X,Y}:\cat{A}(K,X\otimes Y)\to \cat{A}(K,Y\otimes X)$ squaring to the identity. Then \cite[Definition~4.1~(H1)]{cyclic} is the cocycle condition $\psi_{X\otimes Y,Z}\circ \psi_{Y\otimes Z,X}\circ \psi_{Z\otimes X,Y}=\id$ for $X,Y,Z\in\cat{A}$ from \cite[eq. (6.2)]{bd}.

Following~\cite{bjsdualizability}, we call a monoidal category $\cat{C} \in \Rexf$ \emph{p-rigid} if all its projective objects have both a left and right dual which is projective. A central result~\cite[Definition-Proposition 4.1]{bjsdualizability} shows that p-rigid categories are exactly the rigid algebras in $\Rexf$. In particular, the tensor product of a p-rigid monoidal category $\cat{C}$ is exact and hence preserves projective objects. Based on this observation,
 we can define a \emph{p-pivotal} category as a p-rigid category for which in addition the subcategory of projective objects is equipped with a monoidal (non-unital) natural isomorphism $-^{\vee\vee}\to \id $. 
\begin{proposition}\label{PropSymrigid}
	Symmetric rigid algebras in $\Rexf$ are p-pivotal categories.
\end{proposition}
\begin{proof}
	From~\cite[Definition-Proposition 4.1]{bjsdualizability},
	 we know that rigid algebras are exactly p-rigid categories. Hence,
	  we only have to work out the additional structure coming from the symmetry isomorphism $\Sigma$. The natural isomorphism $\Sigma$ of a rigid symmetric algebra is uniquely fixed by its value on the subcategory $\Proj \cat{C}$ of projective objects, which is a natural isomorphism between the two functors
	$(\Proj (\cat{C} \boxtimes  \cat{C}))^{\opp} \to \vect$  sending $P\boxtimes Q$ for $P,Q \in \Proj \cat{C}$ to  
	$\cat{C}(P\otimes Q,I)$ and $\cat{C}(Q\otimes P,I)$, respectively. Thanks to duality,
	this is equivalent to a natural transformation between the functors $\cat{C}(P ,Q^\vee)$ and $\cat{C}( P,{}^\vee Q)$ that via the Yoneda Lemma amount to a natural isomorphism $Q^\vee \to {}^\vee Q$. The condition~\cite[Definition~4.1~(H1)]{cyclic}
	 on $\Sigma$ implies that this is a pivotal structure.      
\end{proof}

\begin{theorem}\label{thmrigid}
	Let $\cat{C}$ be a cyclic associative algebra in $\cat{S}$, and $\Co$ its associated
	open topological field theory.
	Then the following are equivalent:
	\begin{pnum}
		\item There are adjunctions $\Co\left(\begin{gathered}\includegraphics[scale=0.24]{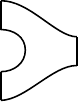}\end{gathered} \right) \dashv \Co\left( \begin{gathered}
		\includegraphics[scale=0.24]{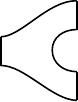}\end{gathered} \right)$ and $\Co\left(\includegraphics[scale=0.4]{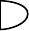}\right) \dashv \Co\left(\includegraphics[scale=0.4]{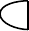}\right) $ such that the natural isomorphisms making $\Co\left( \begin{gathered}
		\includegraphics[scale=0.24]{ocp.pdf}\end{gathered} \right)$ a bimodule functor agree with those coming from the adjunction. \label{thmcharrigidi}
		\item $\cat{C}$ is a symmetric rigid algebra.  \label{thmcharrigidii}
	\end{pnum}
\end{theorem}

If $\cat{S}=\Rexf$, then~\ref{thmcharrigidii} says exactly that $\cat{C}$ is p-pivotal by Proposition~\ref{PropSymrigid}.

\begin{proof}
	If $\cat{C}$ is a symmetric rigid algebra, the monoidal product $\Co\left(\begin{gathered}\includegraphics[scale=0.24]{op.pdf}\end{gathered} \right)$ 
	has a right adjoint $\delta :\cat{C}\to\cat{C}\boxtimes\cat{C}$, and this adjoint  is given by 
	$\Co\left( \begin{gathered}
		\includegraphics[scale=0.24]{ocp.pdf}\end{gathered} \right)$.
	Indeed,  $\Co\left( \begin{gathered}
		\includegraphics[scale=0.24]{ocp.pdf}\end{gathered} \right)$ is given by
	\begin{align}
		(\mu \boxtimes \id_{\cat{C}})\circ (\id_{\cat{C}}\boxtimes \Delta)    
		=(\mu\boxtimes \id_{\cat{C}})\circ (\id_{\cat{C}}\boxtimes (\delta \circ \eta)) \cong \delta \circ \mu \circ (\id_{\cat{C}}\boxtimes \eta) \cong \delta 
	\ . \label{eq23}
	\end{align}
	The second step uses the strong bimodule structure on $\delta$.
	Since $\cat{C}$ is rigid, the exact same argument as in the proof of \cite[Corollary~6.5]{mwansular} gives us the adjunction $\Co\left(\includegraphics[scale=0.4]{1.pdf}\right) \dashv \Co\left(\includegraphics[scale=0.4]{c1.pdf}\right) $.
	Therefore, \ref{thmcharrigidii} implies~\ref{thmcharrigidi}.

Condition~\ref{thmcharrigidi} implies that $\cat{C}$ is rigid: 
Both the multiplication and the unit have a right adjoint by definition, 
and the lax bimodule structure on the right adjoint of the multiplication 
is strong because it is required to be given by isomorphisms.  
\end{proof}

\section{Comparison with the string-net construction}
In this section, we compare
the modular extension $\Co$ of a pivotal finite tensor category $\cat{C}$ with the string-net construction of~\cite{sn}.
This will be the main technical tool for the proof Theorem~\ref{thmpcomp2d2}.

The \emph{string-net construction} of Levin-Wen~\cite{levinwen} can be used to construct categorified open-closed topological field theories. In the semisimple case, it was refined further in~\cite{kirillovsn} and more recently~\cite{bartlett}. A generalization of string-net techniques to pivotal bicategories is given in~\cite{fsy-sn}. 
In~\cite{sn}, string-nets are used to build an open-closed categorified topological field theory from a not necessarily semisimple pivotal finite tensor category $\cat{C}$, and it is shown that its closed part is the Lyubashenko modular functor~\cite{lyubacmp,lyu,lyulex} for $Z(\cat{C})$.
We briefly review the construction of the categorified string-net topological field theory closely following the presentation in~\cite[Section~2]{sn}.

A \emph{graph $\Gamma$ in a surface $\Sigma$} is a finite abstract graph $\Gamma$ together with an embedding of its geometric realization $|\Gamma|$ into $\Sigma$ mapping the endpoints of its leaves (half edges attached to only one vertex), and only those, to the open or closed boundary (see Figure~\ref{Fig:string_net} for an example).    

Let $\cat{C}$ be a {pivotal finite tensor category}.   Informally, the $\cat{C}$-string-net space associated to a surface $\Sigma$ consists of graphs in $\Sigma$ whose edges and vertices are labeled with objects and morphisms of $\Proj \cat{C}$, respectively, modulo relations holding in $\cat{C}$. This idea is implemented via a colimit: We write $\CGraphs(\Sigma)$ for the category whose objects consist of graphs $\Gamma$ in $\Sigma$ together with a labeling of its edges by objects of $\Proj \cat{C}$ and an orientation; we denote the labeling by $\underline{X}$, so that the entire object is a pair $(\Gamma,\underline{X})$. Morphisms are generated by replacing graphs in disks by corollas.
We would like to avoid reproducing here the full definition with all its technical details. 
The reader can find an example of such a replacement with a corolla within a disk in
Figure~\ref{Fig:string_net}. The details are given in~\cite[Definition~2.1]{sn}.    
A \emph{boundary label} for $\Sigma$ is a finite collection of points in the open and closed boundary labeled each
with an object $P\in \Proj\cat{C}$ and a sign $+$ or $-$ 
encoding whether edges ending at the point are incoming or outgoing; we agree that $+$ corresponds to outwards pointing orientation. A $\cat{C}$-labeled graph in $\Sigma$ has a corresponding boundary label. For a fixed boundary label $B$,
we denote by $\CGraphs(\Sigma;B)$ the full subcategory of $\CGraphs(\Sigma)$ of graphs restricting to $B$ on $\partial \Sigma$.  
There is a functor 
\begin{align} 
	\mathbb{E}_{\cat{C}}^{\Sigma,B} :\CGraphs(\Sigma;B) \to \Vect 
\end{align}
from $\cat{C}$-labeled graphs in $\Sigma$ with boundary label $B$
to the category $\Vect$ of
(not necessarily finite-dimensional)
vector spaces over $k$. It
sends a $\cat{C}$-labeled graph in $\Sigma$ with boundary label $B$ to the vector space of compatible labels of the vertices with morphisms in $\cat{C}$. The evaluation of the functor
on morphisms is induced by the graphical calculus for $\cat{C}$. We refer to~\cite[Definition 2.3]{sn} for more details. For the  functor to be well-defined, we need $\cat{C}$ to be pivotal.

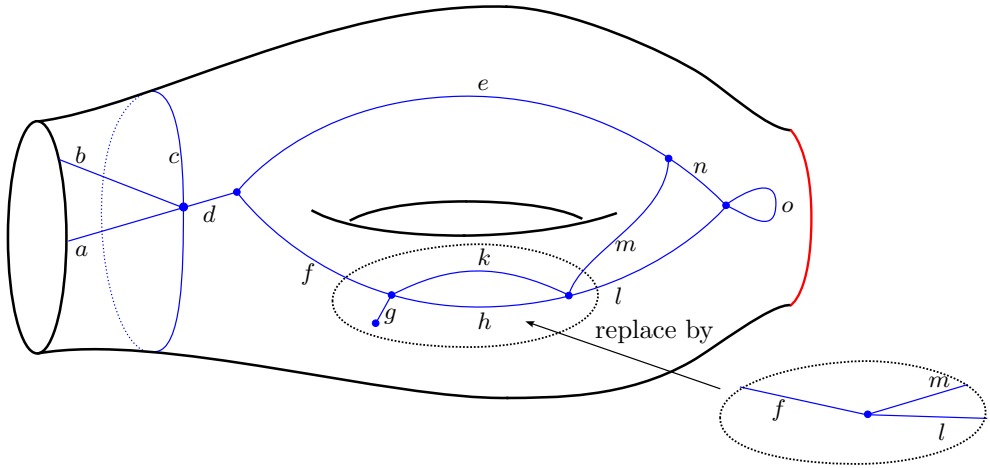
\begin{figure}[h]
	\begin{center}
		\begin{overpic}[scale=0.6]
			{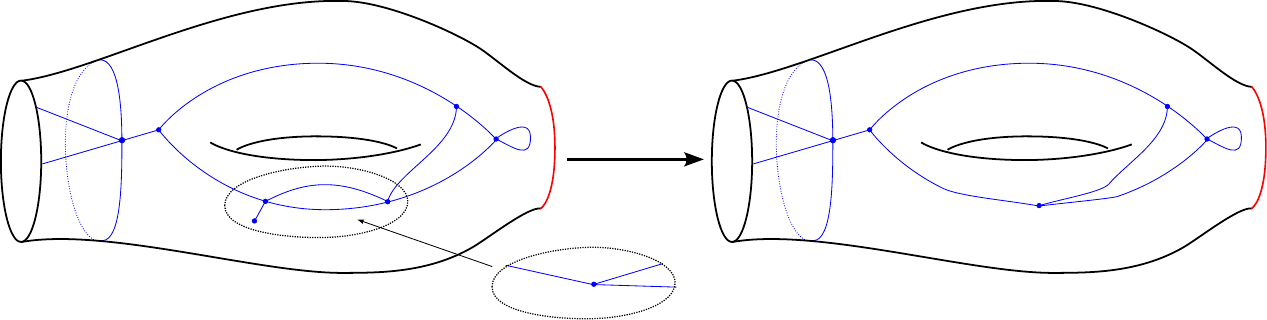}
			\put(3.5,11){\footnotesize$a$}
			\put(3.5,16.5){\footnotesize$b$}
			\put(10,17){\footnotesize$c$}
			\put(10.5,12.5){\footnotesize$d$}
			\put(25,21){\footnotesize$e$}
			\put(14.5,10){\footnotesize$f$}
			\put(20.5,7.5){\footnotesize$g$}
			\put(25,6.8){\footnotesize$h$}
			\put(25,11){\footnotesize$k$}
			\put(36,10){\footnotesize$l$}
			\put(33,14.5){\footnotesize$m$}
			\put(37,16){\footnotesize$n$}
			\put(42.5,13.8){\footnotesize$o$}
			\put(59.7,11){\footnotesize$a$}
			\put(59.7,16.5){\footnotesize$b$}
			\put(66.2,17){\footnotesize$c$}
			\put(67,12.5){\footnotesize$d$}
			\put(81.5,21){\footnotesize$e$}
			\put(71,10){\footnotesize$f$}
			\put(92.5,10){\footnotesize$l$}
			\put(89.5,14.5){\footnotesize$m$}
			\put(93.5,16){\footnotesize$n$}
			\put(98.5,13.8){\footnotesize$o$}
			\put(42,2){\footnotesize$f$}
			\put(48,4.2){\footnotesize$m$}
			\put(49,0.7){\footnotesize$l$}
			\put(26.5,1.3){\small{replace by:}}
		\end{overpic}
	\end{center}
	\caption{An illustration of an object of $\CGraphs(\Sigma)$ and a morphism corresponding to a disk replacement. The labels $a,\dots , o$ are elements of $\Proj \cat{C}$. We have not drawn the orientation associated to every edge. The free boundary of $\Sigma$ is on the right and drawn in red. }
	\label{Fig:string_net}
\end{figure}

\begin{definition}[$\text{\cite[Definition~2.3]{sn} following~\cite[Section~3.2]{fsy-sn}}$]
	Let $\cat{C}$ be a pivotal finite tensor category, $\Sigma$ a surface
	and $B$ a $\cat{C}$-boundary label for $\Sigma$. The \emph{string-net space}
	$\sn_{\cat{C}}(\Sigma;B)$ is defined as the colimit 
	\begin{align}
		\sn_{\cat{C}}(\Sigma;B) := \colimsub{(\Gamma,\underline{X})\in \CGraphs(\Sigma;B)}	\mathbb{E}_{\cat{C}}^{\Sigma,B}(\Gamma,\underline{X}) \ \ .
	\end{align}
\end{definition}

It is explained in \cite[Remark~2.17]{sn} that $\sn_{\cat{C}}(\Sigma;B)$ is isomorphic to the two-dimensional admissible skein module
for $\cat{C}$ and $\Sigma$ in the sense of~\cite{asm}. The mapping class group of $\Sigma$ geometrically acts on the string-net space. 

Using string-net spaces,
 we can assign a $k$-linear category $\snc(S)$ to every oriented 1-dimensional manifold $S$. An object of $\snc(S)$ is a collection of points in $S$ labeled with objects of $\Proj\cat{C}$ such that there is at least one point in every connected component of $S$. The space of morphisms from $\underline{X}$ to $\underline{Y}$ is the vector space $\snc([0,1]\times S,\underline{X}\sqcup \underline{Y})$ with the points in $\underline{X}$ labeled by $-$ and those in $\underline{Y}$ by $+$. Composition is defined by gluing of cylinders. 

The string-net construction naturally extends to a categorified topological field theory with values in the bicategory $\Bimodf$ of $k$-linear pre-finite  categories (see~\cite[Section 5]{sn}), bimodules
(a bimodule from $\cat{C}$ to $\cat{D}$ is a linear functor
 $M: \cat{C}\otimes \cat{D}^{\opp}\to \vect$, sometimes referred to as profunctor), and natural transformations. The composition of two bimodules $M:\cat{C}\otimes \cat{D}^{\opp}\to \vect$ and $M':\cat{D}\otimes \cat{E}^{\opp}\to \vect$ is defined by the coend $\int^{d\in \cat{D}}M(-,d)\otimes M'(d,-)$. 
Explicitly, the open-closed topological field theory built from string-nets assigns  
to a surface $\Sigma$ with incoming boundary $S$ and outgoing boundary $S'$ the bimodule
$
\snc(\Sigma,-,-):\snc(S)\otimes \snc(S')^{\opp} \to \vect $. 
This $\Bimodf$-valued field theory can be turned into a $\Rexf$-valued one by applying the finite free cocompletion, i.e.\ the equivalence
$\Bimodf\simeq \Rexf$ of symmetric monoidal bicategories \cite[Section~5]{sn} sending a pre-finite linear category $\cat{A}$ to the category of linear functors $\cat{A}^\opp \to \vect$,
 and a bimodule to the right exact functor it induces. We denote the $\Rexf$-valued open-closed topological field theory induced by $\snc$ by $\SNC$.

\begin{theorem}\label{thmcomp}
	Let $\cat{C}$ be a pivotal finite tensor category.
	Then the following constructions coincide, up to canonical equivalence,
	 as open $\Rexf$-valued topological field theories:
	\begin{pnum}
		\item The modular extension $\Co$ of $\cat{C}$.\label{thmcompi}
		\item The open part of the open-closed string-net  topological field theory $\SNC$. \label{thmcompiii}
		\end{pnum}
	\end{theorem}

By $\Co$ and the open part of $\SNC$ coinciding up to canonical equivalence we mean that there is an equivalence of modular $\open$-algebras in the sense of \cite[Section~2.4]{cyclic} between them. 

\begin{proof}
	By Theorem~\ref{thmclassopen} it suffices to show that the two
	constructions
	 each produce a $\Rexf$-valued open topological field theory, i.e.\ a $\Rexf$-valued modular $\open$-algebra, whose restriction to disks with intervals embedded in its boundary is $\cat{C}$, seen as a cyclic associative algebra. For~\ref{thmcompi}, this holds by~\eqref{resextequiv0}.
	For $\SNC$ in~\ref{thmcompiii}, it follows from \cite[Theorem~5.9]{sn} 	that we obtain a $\Rexf$-valued	 open-closed modular functor. After restriction to the open part, we obtain an open topological field theory whose restriction to disks is $\cat{C}$, as cyclic associative algebra, as explained in \cite[Section~6]{sn}. This proves the equivalence of~\ref{thmcompi} and~\ref{thmcompiii}.	
	\end{proof}

For the next result, recall that we denote by $L:\cat{C} \to Z(\cat{C})$ the left adjoint to the forgetful functor $U:Z(\cat{C})\to \cat{C}$ from the Drinfeld center $Z(\cat{C})$ of a pivotal finite tensor category $\cat{C}$ to $\cat{C}$.

\begin{theorem}\label{thmpcomp2d2}
	Let $\cat{C}$ be a pivotal finite tensor category and $\Sigma$ a surface 
with $n$ boundary components, at least
one boundary component per connected component,
and exactly one interval embedded in each boundary component.
Then there are for $X_1,\dots,X_n \in \cat{C}$ canonical natural $\Map(\Sigma)$-equivariant
isomorphisms
\begin{align}
	\Co(\Sigma;X_1,\dots,X_n) \ra{\cong} \mathfrak{F}_{Z(\cat{C})}(\Sigma;LX_1,\dots,LX_n)  
\end{align}
between the value of the modular extension $\Co$ of $\cat{C}$ on $\Sigma$ and the space of conformal blocks for $Z(\cat{C})$ on $\Sigma$.
\end{theorem}

\begin{proof}
	We have natural $\Map(\Sigma)$-equivariant isomorphisms
	\begin{align} \Co(\Sigma;X_1,\dots,X_n) \ra{\cong} \SNC(\Sigma;X_1,\dots,X_n)\ra{\cong} \mathfrak{F}_{Z(\cat{C})}(\Sigma;LX_1,\dots,LX_n)  \ , 
	\end{align}
	where the first isomorphism is a consequence of 
	Theorem~\ref{thmcomp}, and the second follows from~\cite[Theorem~7.1]{sn}. 
	\end{proof}

	\begin{corollary}\label{cormcgforGV}
	Let $\Sigma$ be a surface with $n$ boundary components, at least one per connected component, and let $\cat{C}$ be a pivotal	 Grothendieck-Verdier category in $\Lexf$ or $\Rexf$. 
	Then the modular extension  $\Co(\Sigma;X_1,\dots,X_n)$ 
	of $\cat{C}$
	evaluated on $\Sigma$ and boundary labels in $\cat{C}$ carries a representation of the mapping class group of $\Sigma$
	that generalizes the space of conformal blocks, as mapping class group representation, after combination of left and right movers in the following sense:
	If $\cat{C}$ is a modular category, then
	$\Co(\Sigma;X_1,\dots,X_n)$ is $\Map(\Sigma)$-equivariantly isomorphic to the space of conformal blocks $\mathfrak{F}_{\bar{\cat{C}} \boxtimes\cat{C}}(\Sigma;\widetilde X_1,\dots,\widetilde X_n)$ for $\bar{\cat{C}}\boxtimes \cat{C}$ on $\Sigma$ with boundary labels $\widetilde X_i = \int^{Y \in \cat{A}} Y^\vee \boxtimes X_i \otimes Y$ for $1\le i\le n$. 
		\end{corollary}
	
	If $\Sigma$ is connected of genus $g$ with one boundary component (for simplicity), then (in the $\Lexf$-version) one can show
	\begin{align}
		\Co(\Sigma;X)\cong \cat{C} \left(K,X \otimes \left(\otimes^{(4)}\left(\int^{Y,Y' \in \cat{C}} Y\boxtimes Y' \boxtimes DY \boxtimes DY' \right)\right)^{\otimes g}\right) ,\label{eqnformulaChat}
		\end{align}
		using a ribbon graph presentation of $\Sigma$ and the techniques of \cite[Example~7.3]{cyclic}; here $\otimes^{(4)}:\cat{C}^{\boxtimes 4}\to \cat{C}$ is the monoidal product seen as arity four operation.
	
	\begin{proof}[\slshape Proof of Corollary~\ref{cormcgforGV}]
		We obtain the mapping class group representations thanks to Theorem~\ref{thmclassopen}. In order to apply this result, we place exactly one interval on each boundary circle.
	Assume now that $\cat{C}$ is modular:
	Then $\cat{C}$, by the results of~\cite{shimizumodular}, is factorizable in the sense of~\cite{eno-d}, i.e.\ the monoidal product induces a ribbon equivalence $\bar{\cat{C}}\boxtimes \cat{C}\ra{\simeq} Z(\cat{C})$. Now the statement follows from Theorem~\ref{thmpcomp2d2}.
		\end{proof}

\section{The action of the algebra of class functions}
For a finite tensor category $\cat{C}$, one may define an algebra of class functions.
We give here the definition from~\cite{shimizucf}, where also the connection to more classical notions of class functions is made.
We denote by $R:\cat{C}\to Z(\cat{C})$ the right adjoint to the forgetful functor $U:Z(\cat{C}) \to \cat{C}$ from the Drinfeld center $Z(\cat{C})$ of $\cat{C}$ to $\cat{C}$.
Then the
\emph{algebra $\CF(\cat{C})$ of class functions of $\cat{C}$} is the vector space $\cat{C}(\mathbb{A},I)$, where $\mathbb{A}=\int_{X \in \cat{C}} X^\vee \otimes X = URI$ is the canonical end of $\cat{C}$. The algebra structure comes from the fact that $\cat{C}(\mathbb{A},I)=\cat{C}(URI,I)\cong\End_{Z(\cat{C})}(RI)$ can be identified with the endomorphisms of $RI$.

\begin{theorem}\label{thmcf}
		Let $\cat{C}$ be a unimodular pivotal finite tensor category and $\Sigma$ a connected  surface with at least one boundary component,
	$n\ge 0$ boundary intervals labeled by $X_1,\dots,X_n \in \cat{C}$ and
	$\ell \ge 0$ free boundary circles. 
	Then the vector space $\Co(\Sigma;X_1,\dots,X_n)$ associated by the modular extension $\Co$
	of $\cat{C}$ to $\Sigma$ and the labels $X_1,\dots,X_n$
	is naturally a module over $\CF(\cat{C})^{\otimes \ell}$, the $\ell$-th tensor power of the algebra of class functions of $\cat{C}$.
	\end{theorem}

\begin{proof}
We begin with the observation of~\cite[Section~4]{brochierwoike} that the modular extension descends to factorization homology, but instead of using  this fact for the modular extension of cyclic framed $E_2$-algebras, we profit from this principle one dimension lower. Let us give the details:
	Denote by $S$ the free boundary of $\Sigma$
	consisting of $\ell$ circles. 
	For any finite set $J$ and an embedding $\varphi : (0,1)^{\sqcup J} \to S$, 
	we denote by $\Sigma^\varphi$ the surface $\Sigma$ but with the intervals that are part of $\varphi$ as \emph{additional} parametrized intervals.
	The modular extension $\Co$ of $\cat{C}$ then provides  us with a functor
	\begin{align}	
		\Co(\Sigma^\varphi;-) :  \cat{C}^{\boxtimes J} \boxtimes \cat{C}^{\boxtimes n} \to \vect \ . \label{eqnthesefunctors}
		\end{align}
	In other words, we can choose to embed a family of intervals into the free boundary and evaluate with $\Co$. 
	The naturality in the embedding $\varphi$ implies that the functors~\eqref{eqnthesefunctors} descend to factorization homology and give us a functor
	\begin{align}
		\Phi_\cat{C}(\Sigma) : \left(   \int_S \cat{C}\right) \boxtimes \cat{C}^{\boxtimes n} \to \vect 
		\end{align}
	with a canonical isomorphism $\Phi_\cat{C}(\Sigma; \cat{O}_S,-)\cong \Co(\Sigma;-) : \cat{C}^{\boxtimes n}\to \vect$
	for the quantum character sheaf $\cat{O}_S \in \int_S \cat{C}$ from~\cite[Section~5.1]{bzbj}.
	In particular,
	\begin{align} \Co(\Sigma;X_1,\dots,X_n)\cong \Phi_\cat{C}(\Sigma; \cat{O}_S,X_1,\dots,X_n)
		\end{align} as vector spaces.
	But now clearly $\Co(\Sigma;X_1,\dots,X_n)$ becomes a module over the endomorphism algebra
	$\End_{\int_S \cat{C}} (\cat{O}_S)$ of $\cat{O}_S$.

	We will now finish the proof by showing
	that
	\begin{align} \End_{\int_S \cat{C}} (\cat{O}_S) \cong \CF(\cat{C})^{\otimes \ell}   \label{eqnfhcf}
	\end{align}
	as algebras. 
	Whenever $S=S' \sqcup S''$, we have an equivalence $\int_S \cat{C}\simeq \int_{S'}\cat{C}\boxtimes\int_{S''} \cat{C}$ under which $\cat{O}_S$ corresponds to $\cat{O}_{S'}\boxtimes \cat{O}_{S''}$;
	this is just the monoidality of factorization homology with respect to disjoint union.
	Therefore, we can prove~\eqref{eqnfhcf} for $S=\mathbb{S}^1$, i.e.\ $\ell=1$. 
	Up to this point, the proof would apply to any pivotal Grothendieck-Verdier category, but
	$ \End_{\int_{\mathbb{S}^1} \cat{C}} (\cat{O}_{\mathbb{S}^1}) \cong \CF(\cat{C})$
will need that $\cat{C}$ is actually a unimodular pivotal finite tensor category.
Indeed, thanks to $\cat{C}$ being a pivotal finite tensor category and the comparison result Theorem~\ref{thmcomp}, $\Co$ extends to an open-closed modular functor, namely the one built in~\cite{sn} using string-nets, that associates to the circle $\int_{\mathbb{S}^1} \cat{C}$ (that can be understood as finite free cocompletion of the string-net category of $\mathbb{S}^1$) or, equivalently, the Drinfeld center $Z(\cat{C})$.
The equivalence
 	$\int_{\mathbb{S}^1} \cat{C} \simeq Z(\cat{C})$
as linear categories is induced by the left adjoint $L:\cat{C}\to Z(\cat{C})$ to the forgetful functor $U:Z(\cat{C})\to \cat{C}$, see \cite[Theorem~5.9]{sn}.  In particular, under this equivalence $\cat{O}_{\mathbb{S}^1}$ corresponds to $LI \in Z(\cat{C})$, thereby proving that
$
	\End_{\int_{\mathbb{S}^1} \cat{C}} (\cat{O}_{\mathbb{S}^1}) \cong \End_{Z(\cat{C})}(LI) $.
Since $\cat{C}$ is unimodular, $L$ is also right adjoint to $U$~\cite{shimizuunimodular}.
This leaves us with the isomorphism of algebras
$
	\End_{\int_{\mathbb{S}^1} \cat{C}} (\cat{O}_{\mathbb{S}^1}) \cong \End_{Z(\cat{C})}(RI)=\CF(\cat{C})$.
	 This finishes the proof.	\end{proof}

\begin{remark}
	Spaces of conformal blocks have a double role: They are mapping class group representations and skein modules. 
	This can be  efficiently 
	 understood using factorization homology~\cite[Section~4]{brochierwoike}.
	The factorization homology definition of the skein algebra~\cite[Definition~2.6]{skeinfin} tells us that the algebra in~\eqref{eqnfhcf} can be seen as  a one-dimensional version of the  skein algebra. This establishes this very important double role also for spaces of conformal blocks in the open setting.
	\end{remark}

\section{A gluing result involving the partial trace}
In this section,
we use the results of the previous sections to prove a 
gluing result for the spaces of conformal blocks for a Drinfeld center.
More precisely, we decompose the gluing along a circle 
into the gluing along an interval and the partial trace map.

\begin{lemma}\label{lemmapartialtrace}
	Let $\cat{C}$ be a pivotal finite tensor category and $\Sigma \in \open(n)$.
	Then the extension of $\Co$ to an open-closed modular functor induces for any disk $D$ embedded in the interior of $\Sigma$ a natural map
$
	p:	\Co(\Sigma\setminus D;-) \to \Co(\Sigma;-)
	$
	coming from the embedding $\Sigma \setminus D \subset \Sigma$. We call $p$ the partial trace map associated to $\cat{C}$, $\Sigma$ and $D$.
	\end{lemma}

\begin{proof}
	By their very definition string-nets are functorial with respect to embeddings of surfaces. In particular,
	the inclusion $\Sigma \setminus D \subset \Sigma$ induces a map $\SNC(\Sigma \setminus D;-)\to \SNC(\Sigma;-)$. 	
	Now the statement follows from Theorem~\ref{thmcomp}.
	\end{proof}

\begin{remark}\label{rempartialtrace}
Let us justify the use of the term `partial trace' in Lemma~\ref{lemmapartialtrace} by showing that in a specific special case it actually reduces to the partial trace:	Suppose that $\Sigma$ is itself a disk, then $\Sigma \setminus D$ is an annulus $A$. With excision~\cite[Theorem~3.2 \& Proposition~4.2]{sn} for string-nets, we see
(in the $\Lexf$-version) that
	$\SNC(A;-)\cong \cat{C}(I,\mathbb{F}\otimes - )$ for the coend $\mathbb{F}=\int^{X \in \cat{C}}X^\vee \otimes X$ and $\SNC(\Sigma;-)\cong \cat{C}(I,-)$, and that
	 $p$ in this case is
	induced by the map $\mathbb{F}\to I$ coming from the evaluations $X^\vee \otimes X\to I$. 
	If we see a map $f: X \to X \otimes Y$ as a map $I \to X^\vee \otimes X \otimes Y \to \mathbb{F}\otimes Y$ and hence as an element $f \in \SNC(A;Y)$ (denoted by the same symbol by abuse of notation) then $p(f)$ is the partial trace of $f$ taken over $X$.
	Note that the partial trace coincides also with the counit of the adjunction $\Co\left(\begin{gathered}\includegraphics[scale=0.24]{op.pdf}\end{gathered} \right) \dashv \Co\left( \begin{gathered}
		\includegraphics[scale=0.24]{ocp.pdf}\end{gathered} \right)$ that we have at our disposal since $\cat{C}$ is rigid (Theorem~\ref{thmrigid}). 
	\end{remark}

\begin{theorem}\label{thmsewing}
	Let $\cat{C}$ be a pivotal finite tensor category and $\Sigma$ a surface with $n+2$ boundary components, $n\ge 1$, at least one per connected component.
Consider a sewing $\Sigma \to \Sigma'$ that 
identifies two boundary components of $\Sigma$. 
Then for $X,Y_1,\dots,Y_n \in \cat{C}$, the following diagram commutes:
\begin{equation}
	\begin{tikzcd}
		\mathfrak{F}_{Z(\cat{C})}(\Sigma; LX^\vee, LX,LY_1,\dots,LY_n)   \ar[]{rr}{\cong}  \ar[swap]{dd}{\text{$\mathfrak{F}_{Z(\cat{C})}$ evaluated on sewing along circle}} && \Co(\Sigma;X^\vee,X,Y_1,\dots,Y_n) \ar{d}{\text{$\Co$ evaluated on sewing along an interval}}
		\\ && \Co(\Sigma\setminus D ;Y_1,\dots,Y_n) \ar{d}{\text{partial trace map}}
		\\ \mathfrak{F}_{Z(\cat{C})}(\Sigma'; LY_1,\dots,LY_n) \ar[]{rr}{\cong} && \Co(\Sigma ;Y_1,\dots,Y_n)\ , 
	\end{tikzcd} 
\end{equation}
where the horizontal isomorphisms
are the ones from Theorem~\ref{thmpcomp2d2}.
	\end{theorem}
	
	\begin{proof}
		Since $\mathfrak{F}_{Z(\cat{C})}\simeq \SNC$ as modular functors~\cite[Theorem~7.1]{sn},
		the gluing operation for the modular functor $\mathfrak{F}_{Z(\cat{C})}$ is exactly the gluing for string-nets which by the excision result~\cite[Theorem~3.2]{sn} for string-nets is the composition of gluing along the boundary intervals (here by assumption exactly one per boundary component)
		and the embedding of the string-nets; this is depicted schematically in Figure~\ref{Figgluing}.
		\begin{figure}[h!]
			\begin{center}
				\begin{overpic}[scale=0.16]
					{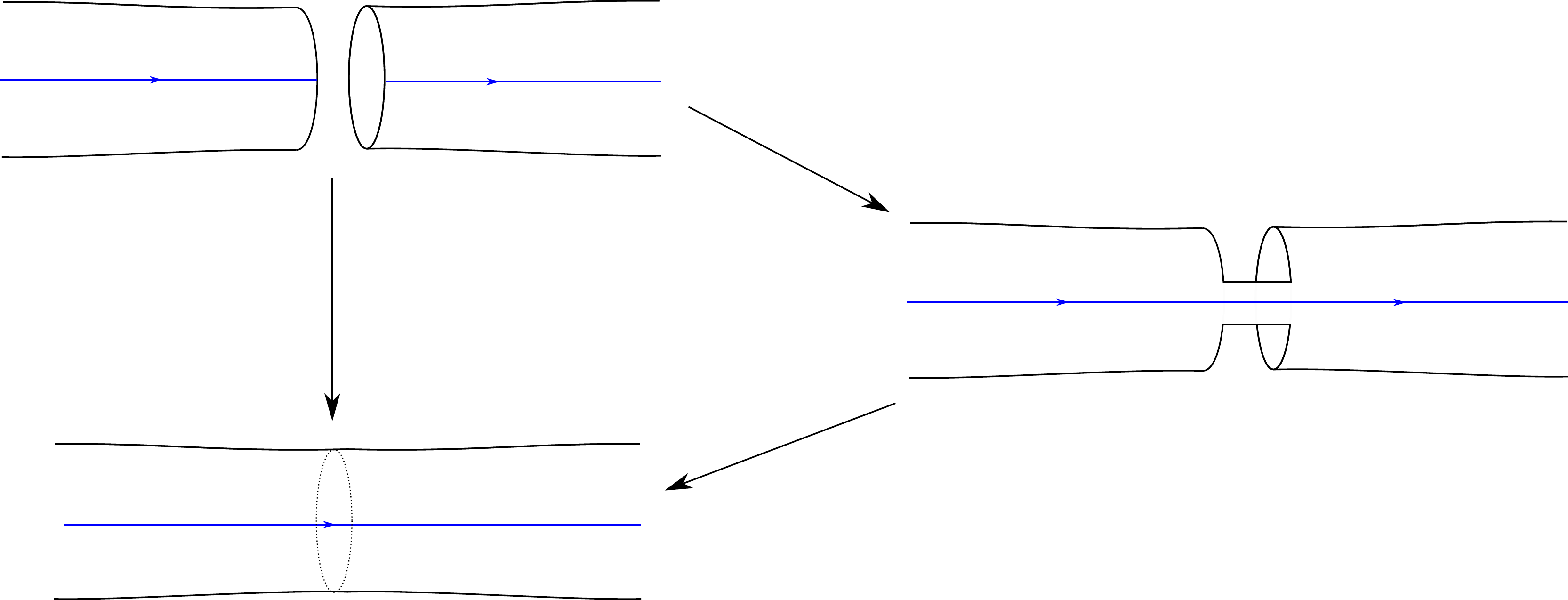}
					\put(15,5.5){\color{blue}$X$}
					\put(10,34){\color{blue}$X$}
					\put(30,34){\color{blue}$X$} 
					\put(70,20){\color{blue}$X$}
					\put(-2,4.5){\small$\dots$}
					\put(43,4.5){\small$\dots$}
					\put(-6.5,33){\small$\dots$}
					\put(44,33){\small$\dots$}
					\put(102,19){\small$\dots$}
					\put(53,19){\small$\dots$}
					\put(52,28.5){gluing along $[0,1]$}
					\put(-8,19){gluing along $\mathbb{S}^1$} 
					\put(52,7.5){embedding of string-nets}
				\end{overpic}
			\end{center}
			\caption{The gluing along a boundary circle can be decomposed into the gluing along a boundary interval and an embedding.}
			\label{Figgluing}
		\end{figure}
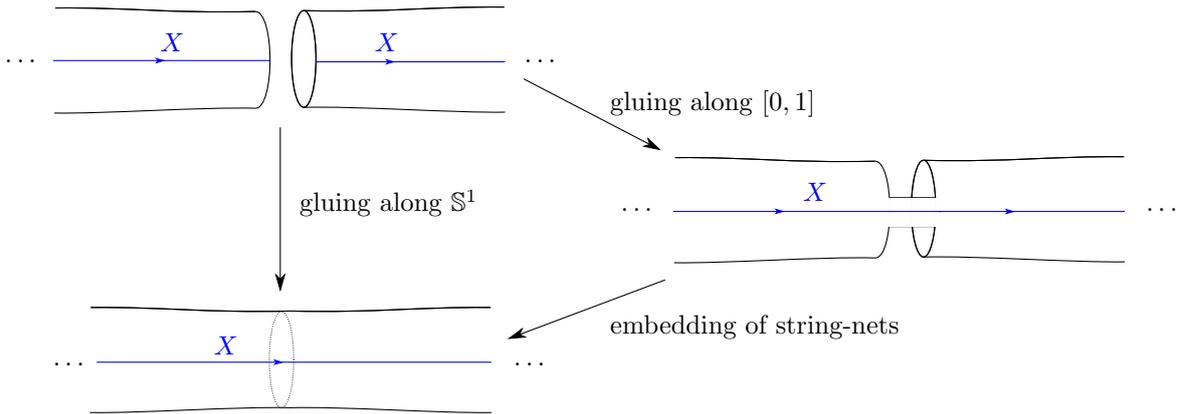 
	
		But the gluing of the string-nets along an interval can, when composed with the isomorphism $\SNC(\Sigma;-)\cong \Co(\Sigma;-)$ from Theorem~\ref{thmcomp},
		agrees with the gluing along intervals for the open modular functor $\Co$ while the embedding of the string-nets along $\Sigma \setminus D \to \Sigma$ yields, when transferred to $\Co$,
		 exactly the partial trace.
		\end{proof}

\noindent \textsc{Perimeter Institute,  N2L 2Y5 Waterloo, Canada} \\[2ex] \noindent \textsc{Université Bourgogne Europe, CNRS, IMB UMR 5584, F-21000 Dijon, France}

\end{document}